\documentclass[11pt]{amsart}
\title[SYMMETRY AND QUASI-CENTRALITY]{
SYMMETRY AND QUASI-CENTRALITY  OF OPERATOR SPACE PROJECTIVE TENSOR PRODUCT
}
\usepackage[ansinew]{inputenc}
\usepackage{yfonts}
\usepackage{calligra}
\usepackage{amscd}
\usepackage{amssymb, latexsym}
\usepackage{mathrsfs}
\usepackage{fancyhdr}
\usepackage{enumerate}
\usepackage{textcomp, phonetic}
\usepackage{setspace}
\usepackage{latexsym}
\usepackage[all]{xy}
\usepackage{tikz}
\theoremstyle{plain}
\newtheorem{thm}{\sc Theorem}[section]
\newtheorem{cor}[thm]{\sc Corollary}
\newtheorem{lem}[thm]{\sc Lemma}

\newtheorem{prop}[thm]{\sc Proposition}

\newtheorem{ex}[thm]{\sc Example}
\newenvironment{pf}{\noindent {\sc Proof:}}{\hfill $\Box$}
\begin{document}
\author[A. Kumar]{
AJAY KUMAR }
\address{Department of Mathematics\\
University of Delhi\\
Delhi -110007\\
India.}
\email{akumar@maths.du.ac.in}
\thanks{}
\author[V. Rajpal]{
Vandana Rajpal }
\address{Department of Mathematics\\
University of Delhi\\
Delhi -110007\\
India.}
\email{vandanarajpal.math@gmail.com}
\begin{abstract}
For $C^{*}$-algebras $A$ and $B$,  the operator space projective tensor product $A\widehat{\otimes}B$ and the
Banach space projective tensor product $A\otimes_{\gamma}B$ are  shown to be symmetric. We also show that  $A\widehat{\otimes}B$
is weakly Wiener algebra. Finally,   quasi-centrality, and the unitary group of $A\widehat{\otimes}B$ are discussed.
\end{abstract}
\maketitle
\section{Introduction}
For a Hilbert space $H$, let $B(H)$ denote the space of all bounded operators on $H$. An operator space $X$ on $H$ is just a closed subspace of $B(H)$~\cite{effros}. For operator spaces $X$ and $Y$, and $u$ an element in the algebraic tensor product $X\otimes Y$,  the operator space projective tensor norm is defined to be\\
 \hspace*{2.3 cm }$\|u\|_{\wedge}=\inf\{\|\alpha\|\|x\|\|y\|\|\beta\|:u=\alpha(x\otimes y)\beta\},$\\
where $\alpha\in M_{1,pq}$, $\beta\in M_{pq,1}$, $x\in M_{p}(X) $ and $y\in M_{q}(Y)$, $p,q\in \mathbb{N}$, and
$x\otimes y=(x_{ij}\otimes y_{kl})_{(i,k),(j,l)}\in M_{pq}(X\otimes Y)$.
The completion of $X\otimes Y$ with respect to this norm is called the operator space
projective tensor product of $X$ and $Y$, and is denoted by $X\widehat{\otimes}Y$. It is well known that,
for $C^{*}$-algebras $A$ and $B$, $A\widehat{\otimes}B$ is a Banach $^{*}$-algebra under the natural involution~\cite{kumar}.
The main objective of this paper is to study properties like symmetry, weak Wiener, quasi-centrality of $A\widehat{\otimes}B$. This $^{*}$-algebra has acquired more attention because
$\|\cdot\|_{\wedge}$ is the largest operator space cross-norm smaller than the Banach space projective norm $\|\cdot\|_{\gamma}$, and there are not sufficient informations available about   $A\otimes_{\gamma} B$.

The notion of symmetry was orginated by  D. A. Raikov in 1946 in the context of algebras with involution  in order to generalize the $C^{*}$-algebra results of Gelfand and Naimark. Later, several authors studied this concept in the context of $L^{1}$-convolution algebra $L^{1}(G)$ of a locally compact group $G$, e.g. see~\cite{leptin}, and~\cite{ludwig}. It is still not known whether the projective tensor product of two symmetric
Banach $^{*}$-algebras is symmetric or not. Kugler in 1979   showed  that if one factor is a type I
$C^{*}$-algebra and other is a symmetric Banach $^{*}$-algebra, then their Banach space projective tensor product
is symmetric~\cite{kugler}.  In  Section 2, we show that the Banach space projective tensor product
  and the operator space projective tensor product  of two $C^{*}$-algebras are symmetric. Symmetry of
 $A\widehat{\otimes}B$ is then used to prove that, for separable $C^{*}$-algebras $A$ and $B$,
$Prim^{*}(A\widehat{\otimes}B)$, the set of kernels of topologically irreducible $^{*}$-representations
of $A\widehat{\otimes}B$, and $Prim(A\widehat{\otimes}B)$, the set of kernels of algebraically irreducible
representations of $A\widehat{\otimes}B$, coincides.    We also show that   $A\widehat{\otimes} B$ is weakly Wiener. Section 3 is devoted to the  quasi-centrality of $A\widehat{\otimes} B$, and $A\otimes_{h} B$, the Haagerup tensor product of $A$ and $B$. Finally, we discuss the unitary group of $A\widehat{\otimes}B$.

Recall that the Haagerup norm on the algebraic tensor product of two $C^{*}$-algebras $A$ and $B$ is defined, for $u\in A\otimes B$, by\\
\hspace*{3 cm}$\|u\|_{h}=\inf\{\|\displaystyle\sum_{i=1}^{n}a_{i}a_{i}^{*}\|^{1/2}\|\displaystyle\sum_{i=1}^{n}b_{i}^{*}b_{i}\|^{1/2}: u=\displaystyle\sum_{i=1}^{n}a_{i}\otimes b_{i} \}$.\\
The Haagerup tensor product $A\otimes_{h}B$ is defined to be the completion of $A\otimes B$ in the norm $\|\cdot\|_{h}$~\cite{effros}.

\section{Symmetry And Weak Wiener Property}
Throughout this paper, for an algebra $A$, $Z(A)$ denotes the center of $A$. By an ideal of $A$
we shall always mean a  two-sided ideal. The set of all closed ideals (resp., proper closed ideals) of $A$ is denoted
by $Id(A)$ (resp., $Id'(A)$), and the set of all primitive ideals of $A$ by $Prim(A)$. By a primitive ideal of $A$, we mean a two sided ideal $I$ of $A$ which is the quotient of a maximal modular left ideal, i.e. $I=L:A=\{a\in A: \;\;aA \subseteq L\}$, for some maximal modular left ideal $L$. $Prim(A)$ is endowed with the hull-kernel topology (hk-topology).  The algebra $A$ is said to be primitive if the zero ideal is a primitive ideal.  In a similar manner, one can define the hk-topology on $Prime(A)$, the space of all prime ideals of $A$, and on $Max(A)$, the set of all maximal modular ideals of $A$.

For a $^{*}$-algebra $A$, let $Prim^{*}(A)$ denote  the set of kernels of topologically irreducible $^{*}$-representations of $A$. Equip $Prim^{*}(A)$ with the hk-topology. If $A$ is a $C^{*}$-algebra, then  $Prim(A)=Prim^{*}(A)$.

A Banach $^{*}$-algebra $A$ is said to be symmetric  if every element of the
form $a^{*}a$ has non-negative spectrum. Ford and Shirali proved that, for a Banach $^{*}$-algebra $A$, this is equivalent to  being a Hermitian, i.e. self-adjoint elements of $A$ have real spectrum. Equivalently, from (~\cite{hermi}, Theorem 1), a Banach $^{*}$-algebra $A$ is symmetric if and only if every proper modular left ideal of $A$ is annihilated by a non-zero positive linear functional. It is well known that every $C^{*}$-algebra is symmetric.

It was shown in (~\cite{ranjana}, Theorem 16) that for a subhomogeneous $C^{*}$-algebra $A$, $A\widehat{\otimes} B$ is symmetric for any $C^{*}$-algebra $B$. In the following, we show that the result is true in general.

\begin{thm}\label{sym1}
For $C^{*}$-algebras $A$ and $B$,  the Banach $^{*}$-algebra $A\widehat{\otimes} B$ is symmetric.
\end{thm}
\begin{pf}
Consider a proper modular left ideal $L$  of $A\widehat{\otimes} B$.  One can assume that $L\neq \{0\}$.  Let  $u\in A\widehat{\otimes} B$ be a right modular unit. We  will show that $\overline{i(L)}$, i.e. min-closure of $i(L)$, is a proper modular left ideal of $A\otimes_{\min} B$ with the right modular unit $i(u)$, where  $i:A\widehat{\otimes}B \to A \otimes_{\min}B$ is the canonical injective $^{*}$-homomorphism~\cite{ranj}. Clearly,  $\overline{i(L)}$ is a left ideal in $A\otimes_{\min} B$.  It is easy to see that
$i(A\widehat{\otimes} B)$ is $\|\cdot\|_{min}$-dense in $A\otimes_{\min} B$. So, for a given $y\in A\otimes_{\min} B$, there exists a sequence $y_{n}\in A\widehat{\otimes} B$ such that $\|y-i(y_{n})\|_{\min}\to 0$. Since $u$ is a right modular unit, so $x-xu\in L $ for all $x\in A\widehat{\otimes} B$. In particular, $i(y_{n})- i(y_{n})i(u)\in i(L)$, and $i(y_{n})-i(y_{n})i(u)$ converges to $y-yi(u)$ in $\|\cdot\|_{\min}$-norm, so $y-yi(u)\in\overline{ i(L)}$, proving that  $\overline{i(L)}$ is a modular left ideal in $A\otimes_{\min} B$. Now, assume that $\overline{ i(L)}=A\otimes_{\min} B$. Since $L$ is a proper modular left ideal in $A\widehat{\otimes} B$, so cl$(L)$ is  also a proper modular left ideal with the right modular unit $u$~\cite{dunc}. Thus $u\notin \text{cl}(L)$ and so,  by the Hahn Banach theorem, there exists $w\in (A\widehat{\otimes} B)^{*}$ such that  $w(u)\neq 0$, and $w(L)=\{0\}$. That is, $L\subseteq \ker(w)$, which  further implies  that $\ker(w)_{\min} = A\otimes_{\min} B$, where $\ker(w)_{\min}$ is the min-closure of $\ker (w)$. Thus $\ker(w)_{\min}$ contains all elementary tensors. Using (~\cite{kumar}, Theorem 6), it follows that  $\ker(w)$ contains all the elementary tensors. Hence $\ker(w)=A\widehat{\otimes} B$. In particular, $w(u)=0$, a contradiction. Hence $\overline{i(L)}$ is a proper modular left ideal of $A\otimes_{\min} B$, and so there exists a non-zero positive linear functional, say $f$ on $A\otimes_{\min} B$ and hence continuous, such that $f(\overline{ i(L)})=\{0\}$, giving that $L\subseteq \ker(f\circ i)$. Since  $i$ is a $^{*}$-homomorphism, so $f\circ i$ is a positive linear functional on $A\widehat{\otimes} B$.  Using the continuity of the functional $f$,  it is easy to verify that  $f\circ i$ is non-zero. Hence  $A\widehat{\otimes} B$ is symmetric.
\end{pf}

\begin{cor}\label{sym11}
Every $P\in Prim(A\widehat{\otimes} B)$ is of the form $P=A\widehat{\otimes} J+I\widehat{\otimes} B$, for some $I\in Prime(A)$ and $J\in Prime(B)$.
\end{cor}
\begin{pf}
By Theorem \ref{sym1} above, and (~\cite{palmer1}, Theorem 11.5.1), we have $Prim(A\widehat{\otimes} B)\subseteq Prim^{*}(A\widehat{\otimes} B)$. So, for a given $P\in Prim(A\widehat{\otimes} B)$ there exist $I\in Prime(A)$ and $J\in Prime(B)$ such that $P=A\widehat{\otimes} J+I\widehat{\otimes} B$, by (~\cite{ranjana}, Theorem 7).
\end{pf}

Now, we proceed to show that the Banach space projective tensor product $A\otimes_{\gamma} B$ is symmetric. Using (~\cite{haag}, Proposition 2), one can verify that the identity map $i$ from $A\otimes_{\gamma} B$ into $A\otimes_{\min} B$ is injective.  The following result for the Banach space projective tensor product of von Neumann algebras can be proved on the similar lines as that in (~\cite{sinc}, Lemma 4.2) for the Haagerup tensor product. However, we outline the proof for the sake of completeness.
\begin{lem}\label{a122}
Let $M$ and $N$ be von Neumann algebras and let $L$ be a closed ideal in $M\otimes_{\gamma}N$. If $1\otimes1 \in L_{\min}$, the closure of $L$ in $\|\cdot\|_{\min}$, then $1\otimes1 \in L$, and hence $L$ equals $M\otimes_{\gamma}N$.
\end{lem}
\begin{pf}
Since $1\otimes1 \in L_{\min}$, so for a given $\delta=\dfrac{1}{2}$, there exists  $u\in L$ such that  $\|i(u)-1\otimes1\|_{\min}< \dfrac{1}{2}$. Let $u=\displaystyle\sum_{j=1}^{\infty} a_{j}\otimes b_{j}$ be a norm convergent representation in $M\otimes_{\gamma}N$ with $\displaystyle\sum_{j=1}^{\infty}\|a_{j}\|\|b_{j}\|< \infty$~\cite{dunc}. By  Dixmier Approximation theorem, there exist sequences $\{v_{j}\}_{j=1}^{\infty}\in Z(M)$, $\{w_{j}\}_{j=1}^{\infty}\in Z(N)$,
$\{\alpha_{n}\}_{n=1}^{\infty}\in P(M)$ and $\{\beta_{n}\}_{n=1}^{\infty}\in P(N)$ such that
\begin{eqnarray}\label{b3}
  \displaystyle\lim_{n\rightarrow \infty} \|\alpha_{n}(a_{j})-v_{j}\|=\displaystyle\lim_{n\rightarrow \infty} \|\beta_{n}(b_{j})-w_{j}\|=0\notag
\end{eqnarray}
for all $j\geq 1$.\\
For each $n\in \mathbb{N}$, define a map $\alpha_{n}\times \beta_{n}: M\times N \to M\otimes_{\gamma}N $ as
$\alpha_{n}\times \beta_{n}(a,b)=\alpha_{n}(a)\otimes \beta_{n}(b)$ for all $a\in M$, and $b\in N$.
So we obtain a unique operator, say $\alpha_{n}\otimes_{\gamma} \beta_{n}$, on $M\otimes_{\gamma}N $ and $L$  is invariant under $\alpha_{n}\otimes_{\gamma} \beta_{n}$ for each $n$. Now as in~\cite{sinc}, the partial sums of $\displaystyle\sum_{j=1}^{\infty} v_{j}\otimes w_{j}$ form a Cauchy sequence in $Z(M)\otimes _{\gamma}Z(N)$ and so we may define an element  $z= \displaystyle\sum_{j=1}^{\infty} v_{j}\otimes w_{j}\in Z(M)\otimes _{\gamma}Z(N)$. Note that $Z(M)\otimes _{\gamma}Z(N)$ is a closed $^{*}$-subalgebra of $M\otimes _{\gamma}N$.
Now for $\epsilon > 0$, choose $m$ such that $\|\displaystyle\sum_{j=m+1}^{\infty} v_{j}\otimes w_{j}\|_{\gamma}\:, \:\|\displaystyle\sum_{j=m+1}^{\infty} a_{j}\otimes b_{j}\|_{\gamma}< \dfrac{\epsilon}{3}$. Thus, for sufficiently large choice of $n$, we have
\begin{align}\label{b4}
\| \alpha_{n}\otimes_{\gamma} \beta_{n}(u)-z\|_{\gamma}&< \epsilon.
\end{align}
So invariance of  $L$ under $\alpha_{n}\otimes_{\gamma} \beta_{n}$  implies that $z\in L $. Consider $\alpha_{n}\otimes \beta_{n}:M\otimes N \to M\otimes_{\min}N$, then clearly $\| \alpha_{n}\otimes \beta_{n}(\displaystyle\sum_{j=1}^{k}m_{j}\otimes n_{j})\|_{\min}\leq \|\displaystyle\sum_{j=1}^ {k}m_{j}\otimes n_{j}\|_{\min}$.  So $\alpha_{n}\otimes \beta_{n}$
can be extended to $M\otimes_{\min}N$. Let the extension be denoted by $ \alpha_{n}\otimes_{\min} \beta_{n} $. It is easy to show that
$i\circ (\alpha_{n}\otimes_{\gamma} \beta_{n})=(\alpha_{n}\otimes_{\min} \beta_{n})\circ i$. So
\begin{align}\label{a21}
 \|(\alpha_{n}\otimes_{\min} \beta_{n})(i(u))-1\otimes1\|_{\min}&\leq \|i(u)-1\otimes1\|_{\min}< \dfrac{1}{2}.
\end{align}
By (\ref{b4}), we have
\begin{align}\label{a23}
 \|i\circ (\alpha_{n}\otimes_{\gamma} \beta_{n})(u)-i(z)\|_{\min}< \epsilon, \text{ for sufficiently large n}.
\end{align}
Thus,
\begin{align}\label{a1009}
 \|i(z)-1\otimes 1\|_{\min}\leq \dfrac{1}{2}< 1.
\end{align}
Now, consider the map $i|_{Z(M)\otimes _{\gamma}Z(N)}$. Clearly, $i|_{Z(M)\otimes _{\gamma}Z(N)}$ is a map from
$Z(M)\otimes _{\gamma}Z(N)$ into $Z(M)\otimes _{\min}Z(N)$. Since $i$ is injective, so is $i|_{Z(M)\otimes _{\gamma}Z(N)}$. By (\ref{a1009}), $i(z)$ is invertible in $Z(M)\otimes _{\min}Z(N)$, so
$0\notin Spec_{Z(M)\otimes _{\min}Z(N)}(i(z)) $. Since $Z(M)\otimes _{\gamma}Z(N)$  is semisimple and regular, so $0\notin Spec_{Z(M)\otimes _{\gamma}Z(N)}(z) $,  by~\cite{kaniuth}. Hence
$z$ is invertible in $Z(M)\otimes _{\gamma}Z(N)$, so there exists $w\in Z(M)\otimes _{\gamma}Z(N)$ such that $zw=wz=1\otimes 1$. Thus, $1\otimes 1\in L$.
\end{pf}

\begin{thm}
For $C^{*}$-algebras $A$ and $B$,  the Banach $^{*}$-algebra $A\otimes_{\gamma} B$ is symmetric.
\end{thm}
\begin{pf}
From (~\cite{Ajay}, Lemma 5.2), we know that the natural embedding $ i_{A}\otimes i_{B}$ of $A\otimes_{\gamma} B$ into $A^{**}\otimes_{\gamma} B^{**}$ is  isometry. So $A\otimes_{\gamma} B$ can be regarded as a closed $^{*}$-subalgebra of $A^{**}\otimes_{\gamma} B^{**}$. Using Lemma \ref{a122}  and arguments similar to that of  Theorem \ref{sym1}, it follows that, for any von Neumann algebras $M$ and $N$, $M\otimes_{\gamma} N$ is symmetric. Thus  $A\otimes_{\gamma} B$ is symmetric.
\end{pf}

Now we look at the structure of $Prim(A\widehat{\otimes}B)$.

\begin{thm}\label{sym2}
For $C^{*}$-algebras $A$ and $B$, let $P\in Prim(A)$ and $Q\in Prim(B)$ then $A\widehat{\otimes}Q+P\widehat{\otimes}B\in Prim(A\widehat{\otimes}B)$.
\end{thm}
\begin{pf}
From (~\cite{ranjana}, Lemma 2) and (~\cite{ranja}, Proposition 3.5), we know that there exists an isometric isomorphism from  $(A\widehat{\otimes}B)/(P\widehat{\otimes}B+A\widehat{\otimes}Q)$ onto $(A/P)\widehat{\otimes}(B/Q)$. Since $P\in Prim(A)$ and $Q\in Prim(B)$, so $A/P$ and $B/Q$ are  primitive $C^{*}$-algebras. Therefore it is enough to show that, for  primitive   $C^{*}$-algebras $A$ and $B$, $A\widehat{\otimes}B$ is a primitive Banach algebra. Suppose first that $A$ and $B$ are unital $C^{*}$-algebras with identities $1_{A}$ and $1_{B}$, respectively. Since  $A$ and $B$ are  primitive   $C^{*}$-algebras, so there exist maximal modular left ideal $M$ in $A$, and $N$ in $B$ such that $\{0\}=M:A$, and  $\{0\}=N:B$. Since the maximal modular left ideal in a Banach algebra is norm closed~\cite{dunc}, so $A\widehat{\otimes}N$ and $M\widehat{\otimes}B$ are well-defined. Now, we show that if $S$ is any left ideal of  $A\widehat{\otimes}B$ containing both $A\widehat{\otimes}N+M\widehat{\otimes}B$, and a non-zero closed ideal $I$ of $A\widehat{\otimes}B$ then $S=A\widehat{\otimes}B$. Since $I$ is a non-zero closed ideal of
$A\widehat{\otimes}B$,  so it would contain a non-zero elementary tensor, say $a\otimes b$, by (~\cite{ranja}, Proposition 3.6). Thus the closed ideal generated by $a\otimes b$ is also contained in $I$, so is contained in $S$. Since $a\otimes b$  is non-zero, so both $a$ and $b$ are non-zeroes. Therefore $a\notin M:A$ and $b\notin N:B$, i.e. $aA\nsubseteq M $ and $bB\nsubseteq N $, which further implies that $<a>$ and $<b>$ are not contained in $M$ and $N$, respectively. Thus there exist $c, d \in A$ and $e, f \in B$ such that $cad \notin M$ and $ebf \notin N$. Now consider $A (cad)+M$ and $B (ebf)+N$. Since $cad \notin M$ and $ebf \notin N$, and so, by the maximality of $M$ and $N$, we have $A (cad)+M=A$ and $B (ebf)+N=B$.  So there exist $g \in A$, $m \in M$, $h \in B$, and $n \in N$ such that $g(cad) + m = 1_{A}$, and $h(ebf) + n = 1_{B}$. Therefore, we have\\
\hspace*{2.8 cm} $1_{A}\otimes 1_{B}=(g(cad) + m) \otimes (h(ebf) + n)$\\
\hspace*{4.4 cm}$=gcad\otimes hebf +gcad\otimes n+m\otimes 1_{B}$\\
\hspace*{4.4 cm}$=(gc\otimes he)(a \otimes b)(d \otimes f) +m\otimes 1_{B}+(1_{A}\otimes n-m\otimes n)$\\
\hspace*{4.4 cm}$\in I+S\subseteq S$\\
Hence $S=A\widehat{\otimes}B$. One can verify that  $A\widehat{\otimes} N+ M\widehat{\otimes} B\subseteq (A\otimes_{h} N+M\otimes_{h} B) \cap A\widehat{\otimes} B$. By the injectivity of the Haagerup norm, $A\otimes_{h} N+M\otimes_{h} B$ is a  left ideal of $A\otimes_{h} B$. Clearly, it is proper, so is contained in some maximal left ideal, say $R$. Then $A\widehat{\otimes} N+ M\widehat{\otimes} B\subseteq R\cap A\widehat{\otimes} B$. It is easy to see that  $R\cap A\widehat{\otimes} B$ is a proper left ideal of $A\widehat{\otimes} B$. So let $T$ be a maximal left ideal of $A\widehat{\otimes}B$ containing $R\cap A\widehat{\otimes} B$. We show that $T:A\widehat{\otimes}B=\{0\}$. Let $0 \neq x\in T:A\widehat{\otimes}B$, then $x A\widehat{\otimes}B\subseteq T$. Thus $<x>\subseteq T$, and so, by the above,  $T=A\widehat{\otimes}B$, a contradiction. Hence, for unital primitive $C^{*}$-algebras $A$ and $B$, $A\widehat{\otimes}B$ is a primitive Banach algebra.

If $A$ has an identity, but $B$ does not have. Consider the unitization $B_{e}$ of  $B$. Clearly,
$B$ is a closed ideal of $B_{e}$. Therefore, $A\widehat{\otimes} B$ is a closed ideal of $A\widehat{\otimes} B_{e}$
by (~\cite{kumar}, Theorem 5). Now using the fact that unitization of  a primitive $C^{*}$-algebra   is  a primitive
$C^{*}$-algebra and the above argument, $A\widehat{\otimes} B_{e}$ is a primitive Banach algebra, i.e. $\{0\}$ is a
primitive ideal of $A\widehat{\otimes} B_{e}$. So, there exists an algebraic irreducible representation $\pi$ of $A\widehat{\otimes} B_{e}$ such that $\ker \pi=\{0\}$. By (~\cite{palmer1}, Theorem 4.1.6),  either $\pi|_{ A\widehat{\otimes} B}$ is irreducible or $\pi| _{A\widehat{\otimes} B}=\{0\}$. If $\pi|_{ A\widehat{\otimes} B}=\{0\}$, then $A\widehat{\otimes} B$ is zero, which is not true, and so $\pi|_{ A\widehat{\otimes} B}$ is irreducible.
Also $\ker \pi|_{ A\widehat{\otimes} B}=\ker \pi \cap A\widehat{\otimes} B=\{0\}$. Hence the result follows in this case. In case, $A$ and $B$ do not have identities, then one may consider $A_{e}\widehat{\otimes} B_{e}$ and apply the similar technique to obtain the result.
\end{pf}

\begin{cor}
For separable $C^{*}$-algebras $A$, $B$, we have $Prim(A\widehat{\otimes} B)=Prim^{*}(A\widehat{\otimes} B)=Prime(A\widehat{\otimes} B)$.
\end{cor}
\begin{pf}
Let $P\in Prim^{*}(A\widehat{\otimes} B)$, then $P=A\widehat{\otimes} J+I\widehat{\otimes} B$ for $I\in Prim^{*}(A)$ and $J\in Prim^{*}(B)$ by (~\cite{ranjana}, Theorem 7), and so, by the Theorem \ref{sym2}, $P\in Prim(A\widehat{\otimes} B)$. Thus $Prim^{*}(A\widehat{\otimes} B)$ $\subseteq Prim(A\widehat{\otimes} B)$. As noted in the  Corollary \ref{sym11}, the reverse containment follows. Note that the other equality follows immediately
from (~\cite{ranjana}, Theorem 6).
\end{pf}

Recall that a Banach algebra $A$  is said to be weakly Wiener if every proper closed ideal of it is contained in  a some primitive ideal.
Note that every $C^{*}$-algebra, and a  Banach algebra with unity is weakly Wiener. Also every symmetric Banach $^{*}$-algebra, which is
weakly Wiener, has the Wiener property, i.e. every proper closed ideal is annihilated by a topologically
irreducible $^{*}$-representation. It has been shown in (~\cite{ranjana}, Theorem 10) that $A\widehat{\otimes} B$ has the  Wiener property.
\begin{cor}\label{sym5}
The Banach algebra $A\widehat{\otimes} B$ is weakly Wiener.
\end{cor}
\begin{pf}
In case  $A$ and $B$ have unity then the result is trivial. It is sufficient to prove the result when   $A$  and $B$ are non-unital.  Consider the unitization $A_{e}$ of  $A$ and $B_{e}$ of  $B$.  Let $I$ be a proper closed ideal of $A\widehat{\otimes} B$. If $I=\{0\}$, the result follows directly by Theorem \ref{sym2}, and so suppose that  $I\neq \{0\}$. Since $A\widehat{\otimes} B$ is a closed ideal of $A_{e}\widehat{\otimes} B_{e}$ with bounded approximate identity, so $I$ is a closed ideal of $A_{e}\widehat{\otimes} B_{e}$. Clearly, $I$ is a proper ideal of $A_{e}\widehat{\otimes} B_{e}$. Thus, there exists an algebraic irreducible representation, say $\pi$,
of  $A_{e}\widehat{\otimes} B_{e}$ such that $I\subseteq \ker \pi$,
which further implies that $I\subseteq \ker \pi| _{A\widehat{\otimes} B}$.
By (~\cite{palmer1}, Theorem 4.1.6), either $\pi|_{ A\widehat{\otimes} B}$ is irreducible
or $\pi|_{ A\widehat{\otimes} B}=\{0\}$. If $\pi| _{A\widehat{\otimes} B}=\{0\}$, then $I$ is zero,
which is not true, and so $\pi|_{ A\widehat{\otimes} B}$ is irreducible. Hence $A\widehat{\otimes} B$ is weakly Wiener.
\end{pf}

Note that the above result can also be proved  for the Banach algebra $A\otimes_{h} B$ using (~\cite{spectral}, Theorem 2.2).

A Banach algebra $A$ is said to be completely regular if $Max(A)$ satisfies: (i) $Max(A)$ is Hausdorff;
(ii) For every $M \in  Max(A) $, there is an open set $U$ in $Max(A)$ containing $M$ such that $k(U)$ is modular ideal.

The following lemma can be proved on the similar lines as done in Lemma 1.3~\cite{arc} for the Haagerup tensor product.
\begin{lem}\label{a123}
Let $I_{0}$ and $J_{0}$ be closed ideals in $C^{*}$-algebras $A$ and $B$, respectively. Let $S\subseteq Id(A)$, $T\subseteq Id(B)$ be such that
$k(S)=I_{0}$ and $k(T)=J_{0}$. Then $A\widehat{\otimes} J_{0}+I_{0}\widehat{\otimes} B=\bigcap \{A\widehat{\otimes} J+I\widehat{\otimes} B : I\in S, J\in T\} $.
\end{lem}
Using the above lemma, and the fact that  the map $\phi$, defined by $\phi(M,N)=A\widehat{\otimes} N+M\widehat{\otimes} B$, is a homeomorphism from $Max(A)\times Max(B)$ onto $Max(A\widehat{\otimes} B)$, which follows from~\cite{ranjana} and~\cite{arc}. We have
\begin{prop}\label{sym9}
For completely regular $C^{*}$-algebras $A$ and $B$, $A\widehat{\otimes} B$ is a completely regular Banach algebra.
\end{prop}

Recall that an algebra  $A$ is said to be  strongly semisimple if $Rad(A)=\{0\}$, where $Rad(A)= \displaystyle\cap_{M\in Max(A)} M$. Using the same argument as in (~\cite{sinc}, Proposition 5.16), it is easy to show that $A\widehat{\otimes} B$ is semisimple. However, $A\widehat{\otimes} B$ is strongly semisimple if and
only if $A$ and $B$ are.  This follows directly by observing that
$Rad(A\widehat{\otimes} B)=A\widehat{\otimes} Rad(B)+Rad(A)\widehat{\otimes} B$, which can be proved
by  using  Lemma \ref{a123} above and (~\cite{ranjana}, Theorem 9).

A Banach algebra  is said to be topologically Tauberian if
every proper closed ideal of it is contained in a some maximal modular ideal. For commutative $C^{*}$-algebras $A$ and $B$,  $A\widehat{\otimes} B$ is a topologically Tauberian Banach algebra by Corollary \ref{sym5}.
\begin{thm}
For topologically Tauberian, completely regular, strongly semisimple  $C^{*}$-algebras $A$ and $B$,  $A\widehat{\otimes} B$ is a topologically Tauberian Banach algebra. Conversely, for $C^{*}$-algebras $A$ and $B$, if $A\widehat{\otimes} B$ is a topologically Tauberian then either $A$ or $B$ is topologically Tauberian.
\end{thm}
\begin{pf}
Since $A$ and $B$ are topologically Tauberian, completely regular, strongly semisimple  $C^{*}$-algebras, so, by (~\cite{palmer1}, Theorem 7.3.4),
$\overline{J_{\infty}(A)}=A$ and $\overline{J_{\infty}(B)}=B$.  Now,  as done in (~\cite{dutta}, Theorem 2.07)
for the Banach space projective tensor product,   we have
$\overline{J_{\infty}(A \widehat{\otimes} B)}=A \widehat{\otimes} B$. Hence the result follows.

For the converse part, suppose that neither $A$ nor $B$ is topologically  Tauberian. So there exist proper closed ideal $I$ in $A$,
and $J$ in  $B$ such that $I\nsubseteq M$ for all $M\in Max(A)$, and $J\nsubseteq N$ for all $N\in Max(B)$. This further implies that
$A \widehat{\otimes} J+I \widehat{\otimes} B \nsubseteq A \widehat{\otimes} N+M \widehat{\otimes} B$. So we get a proper closed ideal
 $A \widehat{\otimes} J+I \widehat{\otimes} B $ of $A\widehat{\otimes} B$, which is not contained in any maximal modular ideal of
 $A\widehat{\otimes} B$, a contradiction. Hence the result.
\end{pf}

In particular, $Z(A)\widehat{\otimes} Z(B)$ is a completely regular  and topologically Tauberian, and Wiener  Banach $^{*}$-algebra.
\section{Quasi-Centrality of $A\widehat{\otimes} B$}
A Banach algebra is said to be quasi-central if no primitive ideal contains its center.  Equivalently, from (~\cite{quasi}, Proposition 1), for a weakly Wiener Banach algebra $A$, if   $A$ is quasi-central then any bounded approximate identity for $Z(A)$ is an approximate identity for $A$, conversely, if $A$ has an approximate identity each element of which belongs to $Z(A)$, then it is quasi-central.
\begin{thm}\label{sym3}
For $C^{*}$-algebras $A$ and $B$, the following  are equivalent \emph{:}\\
\emph{(i)} $A$ and $B$ are quasi-central.\\
\emph{(ii)} $A\widehat{\otimes} B$ is quasi-central.\\
\emph{(iii)} $A\widehat{\otimes} B=Z(A\widehat{\otimes}B)(A \widehat{\otimes}B)$.
\end{thm}
\begin{pf}
(i)$\Rightarrow $ (ii) : Suppose that $A$, $B$ are quasi-central. So $A$ and $B$ both have
approximate identities  each element of which belong to $Z(A)$ and $Z(B)$, respectively. As in ~\cite{sinc},
if there are  more than one approximate identity then we can always index them by the same partially ordered set.
So we let $\{e_{\lambda}\}$, $\{f_{\lambda}\}$ be the  central approximate identities for $A$ and $B$, respectively.
It follows from (~\cite{loy}, Proposition 8.1) that $\{e_{\lambda} \otimes f_{\lambda}\}$ is an approximate identity in  $A\widehat{\otimes} B$. Also, by~\cite{ranj},  there exists a contractive isomorphism $\theta$ from $Z(A)\widehat{\otimes} Z(B)$ onto $Z(A\widehat{\otimes} B)$, so
$e_{\lambda} \otimes f_{\lambda} \in Z(A\widehat{\otimes} B)$. Thus the result follows from Corollary \ref{sym5}.\\
(ii)$\Rightarrow $ (i) : Suppose that $A\widehat{\otimes} B$ is quasi-central. By Corollary \ref{sym5}, $A\widehat{\otimes} B$ is weakly Wiener algebra, so  $A\widehat{\otimes} B$ has central approximate identity. Clearly, $Z(A)\neq \{0\}$, and $Z(B)\neq \{0\}$.  Infact, if any one of them is zero, then $Z(A\widehat{\otimes} B)=\{0\}$ as the map  $\theta$ from $Z(A)\widehat{\otimes} Z(B)$ into $Z(A\widehat{\otimes} B)$ is bijective, which is not true. Now let $\{t_{\nu}\}$ be a central approximate identity for $A\widehat{\otimes} B$.  Let $I:B\to B$ denote the identity operator on $B$. For $\psi\in A^{*}$, consider a map $\psi\times I: A\times B\to B $ defined as\\
\hspace*{4 cm}  $\psi\times I(a,b)=\psi(a)b$. \\
Clearly, $\psi\times I$ is a bilinear map, so there exists a unique linear map $\psi\otimes I$ from $A\otimes B$ into $B$ such that $\psi\otimes I(a\otimes b)=\psi(a)b$. It is easy to see that $\psi\otimes I$    is  $\|\cdot\|_{\wedge}$- continuous, so can be extended to $A\widehat{\otimes} B$. Let the extension be denoted by $\psi\widehat{\otimes} I$.
For  $0\neq a\in  Z(A)$ fixed, define a mapping $T_{a}: A\to A$ as $T_{a}x=xa$. As in (~\cite{loy}, Theorem 8.2),  $\{(T_{a}^{*}\phi\widehat{\otimes}I )(t_{\nu})\}$ is an approximate identity for $B$. A similar argument shows that $A$  also has. Now we show that $(T_{a}^{*}\phi\widehat{\otimes}I )(t_{\nu})b=b(T_{a}^{*}\phi\widehat{\otimes}I )(t_{\nu})$ for all $b\in B$. Let  $t_{\nu}=\displaystyle\sum_{j=1}^{\infty}\lambda_{j}(a_{j}\otimes b_{j})\mu_{j}$ be a norm convergent sum in $A\widehat{\otimes} B$~\cite{effros}. Now using the fact that $t_{\nu}\in Z(A\widehat{\otimes}B)$, and $a\in Z(A)$, we have\\
\hspace*{2.35 cm} $(T_{a}^{*}\phi\widehat{\otimes}I )(t_{\nu})b=(\phi\widehat{\otimes}I )(t_{\nu})(a\otimes b) \\
\hspace*{4.73 cm}  =(\phi\widehat{\otimes}I )(a\otimes b)t_{\nu}=(\phi\widehat{\otimes}I )\displaystyle\sum_{j=1}^{\infty}\lambda_{j}(aa_{j}\otimes bb_{j})\mu_{j}\\
\hspace*{4.73 cm}=\displaystyle\sum_{j=1}^{\infty}\lambda_{j} \phi(aa_{j})bb_{j}\mu_{j}=\displaystyle\sum_{j=1}^{\infty}\lambda_{j} \phi(a_{j}a)bb_{j}\mu_{j}\\
\hspace*{4.73 cm}=b
\displaystyle\sum_{j=1}^{\infty}\lambda_{j} \phi(a_{j}a)b_{j}\mu_{j}=b
\displaystyle\sum_{j=1}^{\infty}\lambda_{j}T_{a}^{*}\phi(a_{j})b_{j}\mu_{j}\\
\hspace*{4.73 cm}=b(T_{a}^{*}\phi\widehat{\otimes}I )(t_{\nu})$.\\
So
$\{(T_{a}^{*}\phi\widehat{\otimes}I )(t_{\nu})\}$ is a central approximate identity for $B$.\\
(ii)$\Rightarrow $ (iii) : It is easy to verify that $A\widehat{\otimes} B$ is a $Z(A \widehat{\otimes}B)$-module. Since the map  $\theta$ from $Z(A)\widehat{\otimes} Z(B)$ onto $Z(A\widehat{\otimes} B)$ is contractive, so  $Z(A\widehat{\otimes} B)$ has a bounded approximate identity with bound 1. By (~\cite{loy}, Theorem 16.2), $(A\widehat{\otimes} B)_{e}=Z(A\widehat{\otimes}B)(A\widehat{\otimes} B)$, where $(A\widehat{\otimes} B)_{e}$ is the essential part of $A\widehat{\otimes} B$, and is given by $\{x\in A\widehat{\otimes} B| \displaystyle\lim_{\lambda} e_{\lambda}x=x\}$, $\{e_{\lambda}\}$ being an approximate identity of $Z(A\widehat{\otimes} B)$. So, it follows easily that $A\widehat{\otimes}B=(A \widehat{\otimes}B)_{e}$.\\
(iii) $\Rightarrow $ (ii) : As done in the above implication,  $(A\widehat{\otimes} B)_{e}=Z(A\widehat{\otimes} B)(A\widehat{\otimes} B)$, so $A\widehat{\otimes} B=(A\widehat{\otimes} B)_{e}$. Then clearly $A\widehat{\otimes} B$ is quasi-central.
\end{pf}

Using the fact that $Z(A\otimes_{h} B)=Z(A)\otimes_{h}Z(B)$~\cite{sinc}, the above result can be proved in a similar way for the Haagerup
tensor product.
\begin{ex}
For a locally compact group $G$, let \emph{$C^{*}(G)$} and \emph{$C^{*}_{r}(G)$} be the group $C^{*}$-algebra, and the reduced group $C^{*}$-algebra  of $G$, respectively.  For locally compact groups $G_{1}$ and $G_{2}$,  \emph{$C^{*}(G_{1})\widehat{\otimes} C^{*}(G_{2})$} \emph{($C^{*}_{r}(G_{1})\widehat{\otimes} C^{*}_{r}(G_{2})$)} is quasi central if and only if $G_{1}$ and $G_{2}$ are \emph{[SIN]}-groups by \emph{(~\cite{losert}, Corollary 1.3)}. For a locally compact  Hausdorff  space $X$, $C_{0}(X)\widehat{\otimes} M_{n}$ is quasi-central. However, no proper closed ideal of $B(H)\widehat{\otimes} B(H)$, for an infinite dimensional  separable Hilbert space $H$, is quasi-central. Similar results hold for the Haagerup tensor product.
\end{ex}

For a unital Banach algebra $A$, let $U(A)=\{u\in A: \;u^{-1}\in A,\; \|u\|= \|u^{-1}\|= 1\}$, and $U_{0}(A)$ denote the subgroup generated by the set $\{e^{i h}:h\in Her(A)\}$, where $Her(A)$ denotes the set of all hermitian elements of $A$. For subsets  $S$ of $A$, and $T$ of $B$, we write $S\otimes T$ for the set $\{a\otimes  b:a\in S,\; b\in T\}$. For unital $C^{*}$-algebras $A$ and $B$, it was shown in
~\cite{kaij} that $U_{0}(A\otimes_{\gamma} B)= U_{0}(A)\otimes U_{0}(B)$, and $U(A\otimes_{h} B)= U(A)\otimes U(B)$~\cite{ranj}. The unitary group of $A\widehat{\otimes} B$
can be characterized in a similar way as follows:

\begin{prop}\label{uni1}
For unital $C^{*}$-algebras $A$ and $B$, $U(A\widehat{\otimes} B)= U(A)\otimes U(B)$.
\end{prop}
\begin{proof}
Suppose that $u\in U(A\widehat{\otimes} B)$. So $\|u\|_{\wedge}= 1$ and $\|u^{-1}\|_{\wedge}= 1$.  Clearly, $i(u)^{-1}=i(u^{-1})$, where  $i:A\widehat{\otimes} B \to A\otimes_{h} B $ is an injective homomorphism~\cite{ranj}. Since $i$ is a continuous map, so $\|i(u)\|_{h}\leq 1$ and $\|i(u)^{-1}\|_{h}\leq 1$. Also, $1 \leq \|i(u)\|_{h} \|i(u)^{-1}\|_{h} $. Thus $\|i(u)\|_{h}=1$. Similarly, $\|i(u)^{-1}\|_{h}=1$. So, by (~\cite{ranj}, Corollary 2), there exist $a\in U(A)$ and $b\in U(B)$ such that $i(u)=a\otimes b$. But $i$ is injective, so $u=a\otimes b$. Hence $U(A\widehat{\otimes} B)\subseteq U(A)\otimes U(B)$. Converse follows easily  from the fact that $\|\cdot \|_{\wedge}$ is a cross norm.
\end{proof}
From~\cite{ranj}, hermitian elements of $A\otimes_{h} B$ (or $A\widehat{\otimes} B $) can be characterized
as $a\otimes 1+1\otimes b$, where $a\in Her(A)$ and $b\in Her(B)$. Hence, one can deduce easily that
 $U_{0}(A\otimes_{h} B)= U_{0}(A)\otimes U_{0}(B)$, and $U_{0}(A\widehat{\otimes} B)= U_{0}(A)\otimes U_{0}(B)$.


\begin{thebibliography}{a}
\bibitem{sinc}Allen, S. D., Sinclair, A. M. and Smith, R. R., The ideal structure of the Haagerup tensor product of $C^{*}$-algebras, \textit{J. Reine Angew. Math.} 442  (1993), 111--148.
      \bibitem{arc}Archbold, R. J., Kaniuth, E., Schlichting, G. and Somerset, D. W. B., Ideal space of the
Haagerup tensor product of $C^{*}$-algebras, \textit{Internat. J. Math.} 8  (1997), 1--29.
\bibitem{dunc} Bonsall, F. F., and Duncan, J., Complete normed algebras, \textit{Berlin, Heidelberg, New York}, 1973.
\bibitem{loy} Doran, R. S., and Wichmann, J., Approximate identities and factorization in Banach modules,
\textit{ Springer-Verlag, Berlin, Heidelberg, New York}, 1979.
\bibitem{dutta} Dutta, T. K., and Goswami, N., Some Intrinsic properties in the tensor product of group algebras, \textit{ Analele  Stiintifice Ale Universit$\breve{a}$tii, "AL. I. Cuzza" IASI Tomul LII, s.I, Matematic$\breve{a}$, f.2}  (2006), 337--349.
\bibitem{effros} Effros, E. G. and Ruan, Z. J., Operator spaces, \textit{Claredon Press-Oxford}, 2000.
 \bibitem{spectral} Feinstein, J. F.  and  Somerset, D. W. B., Spectral synthesis for Banach algebras II, \textit{Math. Zeit.} 239 (2002), 183--213.
     \bibitem{haag} Haagerup, U., The Grothendieck inequality for bilinear forms on $C^{*}$-algebras, \textit{ Adv. Math. 56} (1985), 93--116.
\bibitem{ranj}Jain, R. and Kumar, A., Operator space tensor products of $C^{*}$-algebras, \textit{Math. Zeit.} 260  (2008), 805--811.
    \bibitem{ranja}Jain, R. and Kumar, A., Operator space projective tensor product: Embedding into second dual and ideal structure, arXiv:1106.2644v1 [math.OA].
    \bibitem{ranjana}Jain, R. and Kumar, A., Ideals in operator space projective tensor products of $C^{*}$-algebras, \textit{J. Aust. Math. Soc.} 91 (2011), 275--288.
\bibitem{kaij} Kaijser, S. and Sinclair, A. M., Projective tensor Products of $C^{*}$-algebras, \textit{Math Scand} 55  (1984), 161--187.
\bibitem{kaniuth} Kaniuth, E., A Course in commutative Banach algebras, \textit{Springer-Verlag}, 2009.
   \bibitem{kugler}  Kugler, W., On the symmetry of generalized $L^{1}$-algebras, \textit{Math. Zeit.} 168  (1979), 241--262.
      \bibitem{Ajay} Kumar, A. and  Sinclair, A. M.,  Equivalence of norms on operator space tensor products of $C^{*}$-algebras, \textit{Trans. Amer. Math. Soc.} 350  (1998), 2033--2048.
    \bibitem{kumar}Kumar, A., Operator space projective tensor product of $C^{*}$-algebras,\textit{ Math. Zeit.} 237 (2001), 211--217.
        \bibitem{leptin} Leptin, H. and Poguntke, D.,  Symmetry and nonsymmetry for locally compact groups, \textit{ Journal of Functional Analysis} 33 (1979), 119--134.
         \bibitem{losert} Losert, V., On the center of group $C^{*}$-algebras, \textit{J. Reine Angew. Math.} 554  (2003), 105--138.
        \bibitem{ludwig} Ludwig, J., A Class of symmetric and a class of Wiener group algebras, \textit{ Journal of Functional Analysis} 31 (1979), 187--194.
    \bibitem{hermi} Namsraj, N., On ideals and quotients of Hermitian algebras, \textit{Commentationes Mathematicae Universitatis Carolinae } 18 (1977),  No. 1, 87--91.
        \bibitem{palmer1}Palmer, T. W., Banach algebras and the general theory of $^{*}$-algebras I, II, \textit{Cambridge University Press}, 2001.
   \bibitem{quasi} Takahasi, S., Quasi-centrality of Banach algebras and approximate identities, \textit{ Bull. Fac. Sci., Ibarki Univ. Math.} 15 (1983), 17--18.
\end{thebibliography}
\end{document}